\documentclass[12pt]{amsart}
\usepackage{graphicx} 
\usepackage[margin=1in]{geometry}
\usepackage{amsmath, amsfonts, amsthm, amssymb}
\usepackage{hyperref}  
\usepackage{color} 
\usepackage{array} 
\usepackage{float} 
\usepackage[shortlabels]{enumitem}
\usepackage{ytableau}
\ytableausetup{centertableaux} 

\usepackage{mathtools}
\usepackage{cancel}
\usepackage{tikz} 
\usetikzlibrary{trees} 
\usepackage{pgfplots}
\usetikzlibrary{fillbetween}

\usepackage[english]{babel}
\usepackage[autostyle]{csquotes}

\newtheorem{theorem}{Theorem}[section]
\newtheorem{lemma}[theorem]{Lemma}
\newtheorem{proposition}[theorem]{Proposition}
\newtheorem{corollary}[theorem]{Corollary}

\theoremstyle{definition}
\newtheorem{definition}[theorem]{Definition}
\newtheorem{example}[theorem]{Example}

\title{RSK linear operators and the Vershik-Kerov-Logan-Shepp curve}

\author{Duy Phan and David Xia}
\address{Dept. of Mathematics, U. Illinois at Urbana-Champaign, Urbana, IL 61801, USA}
\email{duyphan2@illinois.edu, davidx3@illinois.edu}

\begin{document}

\maketitle

\begin{abstract}
Stelzer and Yong (2024) studied the Robinson–Schensted–Knuth (RSK) correspondence as a linear operator on the coordinate ring of matrices. They showed that this operator is block diagonal and conjectured that, in a special block, most diagonal entries vanish. We establish this conjecture by identifying these zeros with certain Schensted insertion interactions and analyzing them probabilistically using the Vershik-Kerov-Logan-Shepp Limit Shape Theorem.

\end{abstract}

\section{Introduction}
\subsection{The Schensted insertion algorithm} \label{subsection Schensted}
Given a permutation in one-line notation, the \emph{Schensted insertion} algorithm constructs a \emph{standard Young tableau} by sequentially inserting each number while maintaining increasing order along rows and columns \cite{schensted1961longest}. Starting with an empty tableau, each number is inserted into the first row. If it is larger than all entries in the row, it is appended to the end. Otherwise, it replaces the smallest greater number, which is \enquote{bumped} to the next row, where the process repeats. If a number is bumped from the last row, a new row is created.
\begin{example} 
\label{example intro}
The permutation $31254 \in S_5$ is processed as follows.
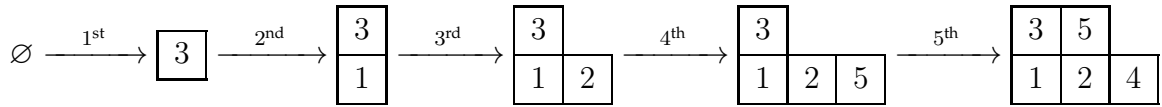
\ytableausetup{boxsize = normal} 
\[
\varnothing
\xrightarrow{\quad 1^{\text{st}}\quad}
\begin{ytableau}
    3
\end{ytableau}
\xrightarrow{\quad 2^{\text{nd}}\quad}
\begin{ytableau}
    3\\
    1
\end{ytableau}
\xrightarrow{\quad 3^{\text{rd}}\quad}
\begin{ytableau}
    3\\
    1 & 2
\end{ytableau}
\xrightarrow{\quad 4^{\text{th}}\quad}
\begin{ytableau}
    3\\
    1 & 2 & 5
\end{ytableau}
\xrightarrow{\quad 5^{\text{th}}\quad}
\begin{ytableau}
    3 & 5\\
    1 & 2 & 4
\end{ytableau}
    \]
\end{example}

We call a bump \emph{vertical} if it moves a number directly upward within the same column; otherwise, we call it \emph{lateral}. Since columns in standard Young tableaux (in French notation) are strictly increasing, a lateral bump always shifts a number strictly to the left. In Example~\ref{example intro}, the second insertion bumps $3$ from the first row and first column to the second row, but stays in the same column, so the bump is vertical. However, the fifth insertion bumps $5$ from the third column to the second column, making it lateral.

\medskip
Let $V_n$ denote the set of all permutations in $S_n$ whose Schensted insertion produces no lateral bumps. As we shall explain, the proof of the Stelzer-Yong conjecture (Theorem~\ref{conj}), follows from Theorem~\ref{main thm} below. Informally, Theorem~\ref{main thm} states that most permutations have a lateral bump in their Schensted insertion. We deduce this probabilistic statement from the celebrated Vershik-Kerov-Logan-Shepp Limit Shape Theorem concerning random partitions under the Plancherel measure.
\begin{theorem} 
\label{main thm}
$\lim_{n\to\infty}\frac{|V_n|}{n!} = 0$.
\end{theorem}

Theorem~\ref{main thm} is closely related to the results of Romik and \'Sniady~\cite{romik2016limit}, who analyzed the bumping route created by the last number in a permutation. Their work shows that the asymptotic shape of this route depends on the value of the last number, leading to a family of limiting curves. Theorem~\ref{main thm} seems derivable from their results (although we could not do so). Instead, we give a self-contained proof based only on the Limit Shape Theorem and elementary arguments.

\subsection{The Stelzer-Yong conjecture}
The RSK correspondence is a well-studied combinatorial bijection between matrices with non-negative integer entries and pairs of semistandard Young tableaux (SSYT) of the same shape. It is a more general version of the Schensted insertion algorithm; restricting the RSK correspondence to permutation matrices and taking the first SSYT of the pair yields exactly the Schensted insertion. For a combinatorial perspective on RSK, see Stanley \cite{stanley1999enumerative}; for its connections to representation theory, see Fulton \cite{fulton1997young}. Stelzer and Yong studied the RSK correspondence as a linear operator on the coordinate ring of matrix space, proving results on its diagonalizability, eigenvalues, trace, and determinant \cite{stelzer2024rsk}. In the same paper, they also conjecture the vanishment of most diagonal entries of a special block in the RSK linear operator (see \cite{stelzer2024rsk}, Conjecture 8.7). As mentioned above, this conjecture is affirmed by Theorem~\ref{main thm} (see Corollary~\ref{cor:conjecture}). To state their conjecture, we first describe the relevant notation from their work.

\medskip
Let $p,q \in \mathbb{N} \coloneqq \{0,1,2, \dots\}$, and let ${\sf Mat}_{p,q}\left(\mathbb{N}\right)$ denote the set of all $p \times q$ matrices with non-negative integer entries. Identify the coordinate ring of the space ${\sf Mat}_{p,q}\left(\mathbb{N}\right)$, given by
\[R_{p,q} \coloneqq \mathbb{C}[z_{i,j}]_{1 \leq i \leq p, 1 \leq j \leq q}.\]
The space $R_{p,q}$ has a \emph{monomial basis} given by
\[z^{\alpha} \coloneqq \prod_{1 \leq i \leq p, 1 \leq j \leq q} z_{i,j}^{\alpha_{i,j}},\]
where $\alpha=(\alpha_{ij}) \in {\sf Mat}_{p,q}(N)$ is the exponent matrix. Doubilet-Rota-Stein introduced another basis for the space $R_{p,q}$ called the \emph{bitableaux basis}, denoted by $\left[P \mid Q\right]$, corresponding to pairs $\left(P,Q\right)$ of semistandard Young tableaux of the same shape \cite{doubilet1974foundations}. Let $\left(P_{\alpha}, Q_{\alpha}\right)$ denote the image of the matrix $\alpha \in {\sf Mat}_{p,q}\left(\mathbb{N}\right)$ under the RSK correspondence. Stelzer and Yong studied a linear operator between these two bases of the space $R_{p,q}$:
\[{\sf RSK} \colon z^{\alpha} \mapsto \left[P_{\alpha} \mid Q_{\alpha}\right].\]

Let $\left(\sigma, \pi\right) \in \mathbb{N}^p \times \mathbb{N}^q$ satisfy $|\sigma|=|\pi|$, where $|\sigma| \coloneqq \sigma_1+\dots+\sigma_p$ and $|\pi| \coloneqq \pi_1+\dots+\pi_q$. Define ${\sf Mat}_{\sigma,\pi}(\mathbb{N})$ as the set of all matrices $\alpha \in {\sf Mat}_{p,q}\left(\mathbb{N}\right)$ whose row and column sums are given by $\left(\sigma, \pi\right)$, i.e.,
\[\sum_{j} \alpha_{i,j}=\sigma_i, \quad 1 \leq i \leq p;\]
\[\sum_{i} \alpha_{i,j}=\pi_j, \quad 1 \leq j \leq q.\]

The \emph{weight} of a Young tableau $T$ is a vector $\left(c_1,c_2,\dots \right)$, where $c_i$ denotes the number of times $i$ appears in $T$. Under the RSK correspondence, if $\alpha \in {\sf Mat}_{\sigma, \pi}$, then $\left( P_{\alpha}, Q_{\alpha} \right)$ has weights $\left( \sigma, \pi \right)$. We recall the notion of the \emph{weight space} (see \cite{stelzer2024rsk}, Section 1.2) associated with $\left(\sigma, \pi\right)$ as
\[R_{\sigma,\pi}= \text{Span} \{z^{\alpha} \mid \alpha \in {\sf Mat}_{\sigma,\pi}(\mathbb{N})\}.\]
The bitableaux $\left[ P \mid Q \right]$, where $\left( P,Q \right)$ has weights $\left( \sigma, \pi \right)$, form a linear basis of $R_{\sigma, \pi}$ (see \cite{stelzer2024rsk}, Lemma 2.5). Consequently, the restriction of the linear operator RSK to $R_{\sigma, \pi}$, denoted by
\[{\sf RSK}_{\sigma, \pi} \colon R_{\sigma, \pi} \to R_{\sigma, \pi},\]
is well-defined. The linear operator ${\sf RSK}_{\sigma,\pi}$ on the subspace $R_{\sigma, \pi}$ can be expressed as a matrix with respect to the monomial basis $\{z^{\alpha} \mid \alpha \in {\sf Mat}_{\sigma, \pi}\}$. The entries of this matrix are given by the function
\[{\sf RSK}_{\sigma,\pi}\left(\beta, \alpha\right)=\left[z^{\beta}\right] \left[P_{\alpha} \mid Q_{\alpha}\right],\]
for all $\alpha, \beta \in {\sf Mat}_{\sigma, \pi}$, where $\left[z^{\beta}\right] \left[P_{\alpha},Q_{\alpha}\right]$ denotes the coefficient of $z^{\beta}$ in $\left[P_{\alpha},Q_{\alpha}\right]$.

\medskip
Let \( 1^n \coloneqq (1, 1, \dots, 1) \); then the block \( {\sf RSK}_{1^n, 1^n} \) plays a significant role in the analysis of \cite{stelzer2024rsk} (see Section 4.1). In particular, Theorem 8.1 of their work shows that the trace of this operator provides leading-term information about the trace of \( {\sf RSK}_{p,q,n} \). Let
\[C_n= \{\alpha \in {\sf Mat}_{1^n,1^n} \mid {\sf RSK}_{1^n,1^n}(\alpha, \alpha)=0\}.\]
We prove the following theorem, originally conjectured in \cite{stelzer2024rsk} (Conjecture 8.7).

\begin{theorem} 
\label{conj}
$\lim_{n \to \infty} \frac{|C_n|}{n!}=1$.
\end{theorem}

\subsection{Organization of the Paper}

A key ingredient in our proof of Theorem~\ref{main thm} is the Limit Shape Theorem, which describes the asymptotic shape of a random Young diagram under the Plancherel measure. Independently established by Vershik–Kerov~\cite{vershik1977asymptotics} and Logan–Shepp~\cite{logan1977variational}, it has deep connections to combinatorics, probability, and representation theory. For historical context and further developments, see Kerov \cite{kerov2003asymptotic}. We revisit this theorem in Section~\ref{subsection limit shape theorem}, using Romik’s reformulation \cite{romik2015surprising}.

\medskip
This paper is organized as follows. In Section~\ref{subsection RSK}, we review the RSK correspondence. Section~\ref{subsection bitableaux} recalls the concept of bitableaux and explains why Theorem~\ref{conj} is equivalent to Theorem~\ref{main thm} (see Corollary~\ref{cor:conjecture}). In Section~\ref{subsection limit shape theorem}, we state the Limit Shape Theorem. Finally, in Section~\ref{section proof}, we provide a self-contained proof of Theorem~\ref{main thm} based solely on the Limit Shape Theorem and elementary arguments.

\section{Preliminaries} 
\label{prelim}
\subsection{RSK} 
\label{subsection RSK}
A \emph{partition} of a positive integer $n$ is a sequence $\lambda = (\lambda_1,\dots,\lambda_k)$ of positive integers such that $\lambda_1  \geq \dots \geq \lambda_k$ and $\lambda_1+ \dots + \lambda_k=n$. The set of partitions of $n$ is denoted by ${\sf Par}(n)$. A \emph{Young diagram} of \emph{shape} $\lambda$ (in French notation) is a collection of boxes arranged in left-justified rows, with row lengths given by $\lambda_1,\dots, \lambda_k$. A \emph{semistandard Young tableau} (SSYT) of shape $\lambda$ is a filling of the boxes of the Young diagram with positive integer entries $a_{i,j}$ satisfying the following two conditions.
 \begin{enumerate}[(i)]
\item $(a_{i,j})$ is weakly increasing along rows, i.e.,
\[a_{i,1} \leq a_{i,2} \leq \ldots \leq a_{i, \lambda_i}, \quad \text{for all} \quad 1 \leq i \leq k.\]
\item $(a_{i,j})$ is strictly increasing along columns, i.e.,
\[a_{1,j}<a_{2,j}< \dots , \quad \text{for all} \quad 1\leq j \leq \lambda_1.\]
 \end{enumerate}
A \emph{standard Young tableau} of shape $\lambda \in {\sf Par}(n)$ is a SSYT with the additional restriction that the entries are precisely $\{1,2,\dots,n\}$, with each appearing exactly once.
\begin{example}
Let $n=7$. The following is a SSYT of the shape $\lambda=(4,2,1)$.
\begin{center}
\ytableausetup{boxsize = normal} 
\begin{ytableau}
    a_{31}\\
    a_{21} & a_{22}\\
    a_{11} & a_{12} & a_{13} & a_{14}
\end{ytableau}
=
\begin{ytableau}
    4\\
    2 & 2\\
    1 & 1 & 1 & 3
\end{ytableau}
\end{center}
\end{example}

For a matrix $\alpha \in {\sf Mat}_{p,q}(\mathbb{N})$, we can construct a \emph{biword} of $\alpha$, denoted by ${\sf biword}(\alpha)$, by reading the entries $\alpha_{i,j}$ down the columns from left to right and recording them as $(ii \dots i \mid jj \dots j)$, where $\alpha_{i,j}$ is associated with the numbers $i$ and $j$ on each side.

\begin{example}
Consider
$\alpha =
\begin{pmatrix}
    1 & 0 & 2\\
    0 & 2 & 0\\
    1 & 1 & 0
\end{pmatrix}$. Then, we have
\[{\sf biword}(\alpha)=(1322311 \mid 1122233).\]
\end{example}

Given ${\sf biword}(\alpha)=(a_1,a_2,\dots ,a_n \mid b_1,b_2, \dots ,b_n)$, we can construct a pair of semistandard Young tableaux, known as the \emph{insertion} and \emph{recording} tableaux, using an algorithm called \emph{row insertion}. Row insertion is a more general form of Schensted insertion, and it proceeds as follows: Start with an empty insertion tableau and an empty recording tableau. At each step $i$, insert $a_i$ into the insertion tableau, beginning with the bottom row. If the row is empty, create a new box containing $a_i$. If the row is not empty and $a_i$ is larger than all entries in the row, create a new box at the end containing $a_i$. Otherwise, replace the leftmost number $x>a_i$ in the row and \emph{bump} $x$ into the row above in the same manner, repeating this process as necessary for all rows. Whenever a new box is created in the insertion tableau, create a corresponding box in the recording tableau in the same position, containing $b_i$. After $n$ steps, the resulting pair of tableaux is denoted $\left(P_{\alpha},Q_{\alpha}\right)$.

\begin{example} 
\label{rsk example}
Consider the ${\sf biword}(\alpha)=(1322311 \mid 1122233)$. The algorithm for constructing $(P_{\alpha},Q_{\alpha})$ is as follows.
\[
\left(\varnothing, \varnothing\right),
\left(\begin{ytableau}
    1
\end{ytableau},
\begin{ytableau}
    1
\end{ytableau}\right), \left(\begin{ytableau}
    1 & 3
\end{ytableau},
\begin{ytableau}
    1 & 1
\end{ytableau}\right), 
\left(
\begin{ytableau}
    3\\
    1 & 2
\end{ytableau},
\begin{ytableau}
    2\\
    1 & 1\\
\end{ytableau}
\right), \quad \quad
\]

\[
\left(
\begin{ytableau}
    3\\
    1 & 2 & 2
\end{ytableau},
\begin{ytableau}
    2\\
    1 & 1 & 2\\
\end{ytableau}
\right),
\left(
\begin{ytableau}
    3\\
    1 & 2 & 2 & 3
\end{ytableau},
\begin{ytableau}
    2\\
    1 & 1 & 2 & 2
\end{ytableau}
\right), \quad \quad \quad
\]

\[
\left(
\begin{ytableau}
    3\\
    2\\
    1 & 1 & 2 & 3
\end{ytableau},
\begin{ytableau}
    3\\
    2\\
    1 & 1 & 2 & 2\\
\end{ytableau}
\right),
\left(
\begin{ytableau}
    3\\
    2 & 2\\
    1 & 1 & 1 & 3
\end{ytableau},
\begin{ytableau}
    3\\
    2 & 3\\
    1 & 1 & 2 & 2\\
\end{ytableau}
\right) = (P_{\alpha},Q_{\alpha}).
\]
\end{example}

\subsection{Bitableaux and zero coefficients}
\label{subsection bitableaux}
For a pair $(P,Q)$ of SSYT of the same shape $\lambda$, with entries $\{a_{i,j}\}$ and $\{b_{i,j}\}$ respectively, we construct minors $\Delta_j$ for each $1 \leq j \leq \lambda_1$ as follows. Consider an infinite matrix with entries $\{z_{i,j}\}_{i,j \geq 1}$. For each $j$, we extract a square submatrix whose rows are indexed by the entries of column $j$ of $P$ and whose columns are indexed by the entries of column $j$ of $Q$. The determinant of this submatrix defines the minor $\Delta_j$. The \emph{bitableau} of $(P,Q)$ is then defined as the product of all minors $\Delta_j$, denoted by $\left[P \mid Q \right]$.

\begin{example}
Considering the pair of SSYT in Example~\ref{rsk example}, we have
\[[P_{\alpha} \mid Q_{\alpha}]= 
\begin{vmatrix}
    z_{11} & z_{12} & z_{13}\\
    z_{21} & z_{22} & z_{23}\\
    z_{31} & z_{32} & z_{33}
\end{vmatrix}.
\begin{vmatrix}
    z_{11} & z_{13}\\
    z_{21} & z_{23}
\end{vmatrix}.
\begin{vmatrix}
    z_{12}
\end{vmatrix}.
\begin{vmatrix}
    z_{32}
\end{vmatrix}.
\]
\end{example}

For each $\alpha \in {\sf Mat}_{1^n,1^n}$, the matrix $\alpha$ can be viewed as a permutation matrix represented by $w_{\alpha}=w_1 \dots w_n \in S_n$, where the entries satisfy $\alpha_{w_i,i}=1$ for all $i \in [n]$ and all other entries are zero. In this case, ${\sf biword}(\alpha)=(w_1, \dots, w_n \mid 1, \dots, n)$, so $(P_{\alpha},Q_{\alpha})$ is a pair of standard Young tableaux with weights $(1^n,1^n)$. The algorithm that constructs $P_{\alpha}$ is precisely the Schensted insertion of the permutation $w_{\alpha}$.

\begin{lemma} 
\label{equivalence lemma}
For $\alpha \in {\sf Mat}_{1^n,1^n}$, we have ${\sf RSK}_{1^n,1^n}(\alpha, \alpha)=0$ if and only if $w_{\alpha}$ has a lateral bump in its Schensted insertion.
\end{lemma}

\begin{proof}
We observe that  
\[
z^\alpha = z_{w_1,1} z_{w_2,2} \dots z_{w_n,n}
\]
is a monomial of degree \( n \), and  
\[
\left[ P_{\alpha} \mid Q_{\alpha} \right] = \Delta_1 \Delta_2 \dots
\]
is a homogeneous polynomial of degree \( n \). Suppose that \( w_{\alpha} \) undergoes a lateral bump during its Schensted insertion. We consider the last insertion \( w_i \) that causes a lateral bump. Then \( w_i \) is inserted at the bottom of some \( k^{\text{th}} \) column of \( P_{\alpha} \), and the new box containing \( i \) is created in the \( l^{\text{th}} \) column of \( Q_{\alpha} \), where \( l < k \). By the choice of \( w_i \), any subsequent insertion \( w_{i+1}, \dots, w_n \) undergoes only vertical bumps, so \( w_i \) always remains in the \( k^{\text{th}} \) column of \( P_{\alpha} \). Since \( w_i \) is in the \( k^{\text{th}} \) column of \( P_{\alpha} \) while \( i \) is in the \( l^{\text{th}} \) column of \( Q_{\alpha} \), the variable \( z_{w_i,i} \) does not appear in any minor, implying that the coefficient of \( z^{\alpha} \) in \( \left[ P_{\alpha} \mid Q_{\alpha} \right] \) is zero.

Conversely, if $w_{\alpha}$ has no lateral bumps, then by a similar argument as above, each $z_{w_i,i}$ for $i \in [n]$ appears in a unique minor $\Delta_k$. Since $\left[ P_{\alpha} \mid Q_{\alpha} \right]$ is a homogeneous polynomial of degree $n$, constructing the degree-$n$ monomial $z^{\alpha}$ requires selecting exactly one subterm with coefficient $\pm 1$ from each $\Delta_k$. Here, a subterm refers to a product of some $z_{w_i,i}$ for $i \in [n]$. Therefore, the coefficient of $z^{\alpha}$ is $\pm 1$. \qedhere
\end{proof}

\begin{corollary}
\label{cor:conjecture}
$\lim_{n\to\infty}\frac{|V_n|}{n!} = 0$ if and only if $\lim_{n \to \infty} \frac{|C_n|}{n!}=1$.
\end{corollary}

\begin{proof}
    Lemma~\ref{equivalence lemma} establishes that Theorem~\ref{conj} is equivalent to Theorem~\ref{main thm}. \qedhere
\end{proof}

\subsection{Limit Shape Theorem}
\label{subsection limit shape theorem}

The restriction of the RSK correspondence to pairs of standard Young tableaux is known as the Schensted correspondence, which provides a bijection between permutations and pairs of standard Young tableaux of the same shape. For $\lambda \in {\sf Par}(n)$, let $f_{\lambda}$ denote the number of standard Young tableaux of shape $\lambda$. Define the function
\[{\sf sh} \colon S_n \to {\sf Par}(n), \quad w \mapsto \lambda_w\]
where $\lambda_w$ is the shape of the tableau obtained by applying the Schensted insertion algorithm to $w$. Then, by the Schensted correspondence, we obtain
\[|{\sf sh}^{-1}(\lambda)|=|\{w \in S_n \mid {\sf sh}(w)=\lambda \}|=f_{\lambda}^2.\]
Now, consider a \emph{uniform probability measure} $\mathbb{P}_n$ on $S_n$, where
\[\mathbb{P}_n(w) \coloneqq \frac{1}{n!}, \quad \text{for all} \quad w \in S_n.\]
The \emph{Plancherel measure} $\mathbb{P}_n^*$ on ${\sf Par}(n)$ is defined as the pushforward measure of the uniform measure on $S_n$ via ${\sf sh}$ (see \cite{johansson2001discrete}), i.e.,
\[\mathbb{P}_n^* (\lambda) \coloneqq \mathbb{P}_n \left( {\sf sh} \left( w \right)=\lambda \right)=\frac{f_{\lambda}^2}{n!}, \quad \text{for all} \quad \lambda \in {\sf Par}(n).\]

We recall Romik’s reformulation of the Vershik-Kerov-Logan-Shepp Limit Shape Theorem in the standard Cartesian coordinate system from \cite{romik2015surprising}.

\begin{definition} The \emph{limit shape} $\Gamma$ in the $\mathbb{R}^2$ plane is the region (see Figure~\ref{limitshape}) bounded by the coordinate axes and the following parametric curve:
\[
\gamma = \left(\gamma_x, \gamma_y\right) \colon \left[-\frac{\pi}{2}, \frac{\pi}{2}\right] \longrightarrow \mathbb{R}^2,
\]
\[
\gamma_x\left(\theta\right) = \left( \frac{2\theta}{\pi} + 1 \right) \sin\theta + \frac{2}{\pi}\cos\theta,
\]
\[
\gamma_y\left(\theta\right) = \left( \frac{2\theta}{\pi} - 1 \right) \sin\theta + \frac{2}{\pi}\cos\theta.
\]
\end{definition}

\begin{figure}
    \centering
    \begin{tikzpicture}
        \begin{axis}[
    axis lines= middle,
    xtick={2}, 
    ytick={2}, 
    xmin=-0.25, xmax=2.25,
    ymin=-0.25, ymax=2.25,
    axis equal,
    width=6cm, height=6cm,
    ] 
    \addplot[name path=curve1, domain=-pi/2:pi/2, samples=100, thick, variable=\t] 
        ({((2*\t/pi) + 1) * sin(deg(\t)) + (2/pi) * cos(deg(\t))}, 
         {((2*\t/pi) - 1) * sin(deg(\t)) + (2/pi) * cos(deg(\t))}); 

    \addplot[name path=y1, domain=0:1.8,samples=10,
    thin,variable=\t]
        ({\t},
        {0.1});
    \addplot[name path=y2, domain=0:1.5,samples=10,
    thin,variable=\t]
        ({\t},
        {0.2});
    \addplot[name path=y3, domain=0:1.1,samples=10,
    thin,variable=\t]
        ({\t},
        {0.3});
    \addplot[name path=y4, domain=0:0.9,samples=10,
    thin,variable=\t]
        ({\t},
        {0.4});
    \addplot[name path=y5, domain=0:0.9,samples=10,
    thin,variable=\t]
        ({\t},
        {0.5});
    \addplot[name path=y6, domain=0:0.7,samples=10,
    thin,variable=\t]
        ({\t},
        {0.6});
    \addplot[name path=y7, domain=0:0.7,samples=10,
    thin,variable=\t]
        ({\t},
        {0.7});
    \addplot[name path=y8, domain=0:0.6,samples=10,
    thin,variable=\t]
        ({\t},
        {0.8});
    \addplot[name path=y9, domain=0:0.4,samples=10,
    thin,variable=\t]
        ({\t},
        {0.9});
    \addplot[name path=y11, domain=0:0.3,samples=10,
    thin,variable=\t]
        ({\t},
        {1});
    \addplot[name path=y12, domain=0:0.3,samples=10,
    thin,variable=\t]
        ({\t},
        {1.1});
    \addplot[name path=y13, domain=0:0.2,samples=10,
    thin,variable=\t]
        ({\t},
        {1.2});
    \addplot[name path=y14, domain=0:0.2,samples=10,
    thin,variable=\t]
        ({\t},
        {1.3});
    \addplot[name path=y15, domain=0:0.2,samples=10,
    thin,variable=\t]
        ({\t},
        {1.4});
    \addplot[name path=y16, domain=0:0.1,samples=10,
    thin,variable=\t]
        ({\t},
        {1.5});
    \addplot[name path=y17, domain=0:0.1,samples=10,
    thin,variable=\t]
        ({\t},
        {1.6});
    \addplot[name path=x1, domain=0:1.6,samples=10,thin,variable=\t]
        ({0.1},
        {\t});
    \addplot[name path=x2, domain=0:1.4,samples=10,thin,variable=\t]
        ({0.2},
        {\t});
    \addplot[name path=x3, domain=0:1.1,samples=10,thin,variable=\t]
        ({0.3},
        {\t});
    \addplot[name path=x4, domain=0:0.9,samples=10,thin,variable=\t]
        ({0.4},
        {\t});
    \addplot[name path=x5, domain=0:0.8,samples=10,thin,variable=\t]
        ({0.5},
        {\t});
    \addplot[name path=x6, domain=0:0.8,samples=10,thin,variable=\t]
        ({0.6},
        {\t});
    \addplot[name path=x7, domain=0:0.7,samples=10,thin,variable=\t]
        ({0.7},
        {\t});
    \addplot[name path=x8, domain=0:0.5,samples=10,thin,variable=\t]
        ({0.8},
        {\t});
    \addplot[name path=x9, domain=0:0.5,samples=10,thin,variable=\t]
        ({0.9},
        {\t});
    \addplot[name path=x10, domain=0:0.3,samples=10,thin,variable=\t]
        ({1},
        {\t});
    \addplot[name path=x11, domain=0:0.3,samples=10,thin,variable=\t]
        ({1.1},
        {\t});
    \addplot[name path=x12, domain=0:0.2,samples=10,thin,variable=\t]
        ({1.2},
        {\t});
    \addplot[name path=x13, domain=0:0.2,samples=10,thin,variable=\t]
        ({1.3},
        {\t});
    \addplot[name path=x14, domain=0:0.2,samples=10,thin,variable=\t]
        ({1.4},
        {\t});
    \addplot[name path=x15, domain=0:0.2,samples=10,thin,variable=\t]
        ({1.5},
        {\t});
    \addplot[name path=x16, domain=0:0.1,samples=10,thin,variable=\t]
        ({1.6},
        {\t});
    \addplot[name path=x17, domain=0:0.1,samples=10,thin,variable=\t]
        ({1.7},
        {\t});
    \addplot[name path=x18, domain=0:0.1,samples=10,thin,variable=\t]
        ({1.8},
        {\t});
        \end{axis}
    \end{tikzpicture}
    \caption{The limit shape \( \Gamma \) and a \( \frac{1}{10} \)-rescaled partition \( \lambda \in {\sf Par}(100) \).}
    \label{limitshape}
\end{figure}
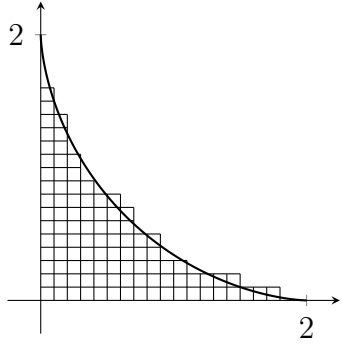

\begin{definition}
    Given a partition $\lambda = \left(\lambda_1, \cdots , \lambda_k\right)$, we define its \emph{planar set} as
    \[S_{\lambda} = \bigcup_{1\leq i\leq k, 1\leq j \leq \lambda_i}\left([i-1, i]\times[j-1,j]\right)\subseteq \mathbb{R}^2.\]
\end{definition}

\begin{example}
Consider $w=25143 \in S_5$. Then, ${\sf sh}(w) = \left(2,2,1\right)$. Its corresponding planar set is shown in Figure~\ref{fig:filled_squares}.
\end{example}

\begin{theorem}[Limit Shape Theorem, Theorem 1.26 of \cite{romik2015surprising}] 
\label{limit shape theorem}
For a fixed $\epsilon \in \left(0,1\right)$, under the Plancherel measure $\mathbb{P}_n^*$ on ${\sf Par}(n)$, we have
\[\mathbb{P}_n^* \left( \left(1-\epsilon\right) \Gamma \subseteq \frac{1}{\sqrt{n}} S_{\lambda} \subseteq \left(1+\epsilon\right) \Gamma \right) \longrightarrow 1 \quad \text{as} \quad n\longrightarrow \infty.\]
\end{theorem}

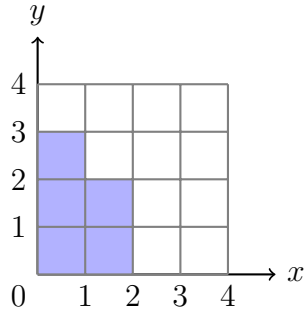
\begin{figure}[H]
\centering
\begin{tikzpicture}[x=1.5em, y=1.5em]

    \draw[thick,->] (0,0) -- (5,0) node[right] {$x$};
    \draw[thick,->] (0,0) -- (0,5) node[above] {$y$};
    
    \fill[blue!30] (0,0) rectangle (1,1); 
    \fill[blue!30] (1,0) rectangle (2,1); 
    \fill[blue!30] (1,1) rectangle (2,2); 
    \fill[blue!30] (0,1) rectangle (1,2); 
    \fill[blue!30] (0,2) rectangle (1,3); 

    \draw[thick, gray] (0,0) grid[xstep=1, ystep=1] (4,4);
    
    \node[below left] at (0,0) {$0$};
    \node[below] at (1,0) {$1$};
    \node[below] at (2,0) {$2$};
    \node[below] at (3,0) {$3$};
    \node[below] at (4,0) {$4$};
    \node[left] at (0,1) {$1$};
    \node[left] at (0,2) {$2$};
    \node[left] at (0,3) {$3$};
    \node[left] at (0,4) {$4$};
    
\end{tikzpicture}
\caption{Planar set of $\lambda=(2,2,1)$.}
\label{fig:filled_squares}
\end{figure}

\section{Proof of Theorem~\ref{main thm}}
\label{section proof}
\begin{definition}
An \( n \)-word \( a_1 a_2 \dots a_n \), where each \( a_i \in \mathbb{R} \), is called \emph{injective} if
\[
a_i \neq a_j \quad \text{for all} \quad 1 \leq i < j \leq n.
\]
\end{definition}

\begin{definition}
The \emph{flattening operation} {\sf Flat} on the set of injective \( n \)-words is given by
\[
{\sf Flat} \colon a_1 a_2 \dots a_n \mapsto b_1 b_2 \dots b_n,
\]
where $b_i \coloneqq \left| \left\{ j \in [n] \mid a_j \leq a_i \right\} \right|$. 
Evidently, \( b_1 b_2 \dots b_n\) is a permutation on $[n]$.
\end{definition}

\begin{definition}
Two injective \( n \)-words \( a_1 a_2 \dots a_n \) and \( c_1 c_2 \dots c_n \) are said to have the same \emph{relative order} if, for all \( 1 \leq i < j \leq n \), we have \( a_i < a_j \) if and only if \( c_i < c_j \). In this case, we also write \( a_1 a_2 \dots a_n \simeq c_1 c_2 \dots c_n \).
\end{definition}

\begin{lemma}
\label{lemma:flat}
For any injective \( n \)-word \( a_1 a_2 \dots a_n \), we have
\[
a_1 a_2 \dots a_n \simeq {\sf Flat}(a_1 a_2 \dots a_n).
\]
Moreover, for two injective \( n \)-words \( a_1 a_2 \dots a_n \) and \( c_1 c_2 \dots c_n \), we have
\[
a_1 a_2 \dots a_n \simeq c_1 c_2 \dots c_n \quad \text{if and only if} \quad
{\sf Flat}(a_1 a_2 \dots a_n) = {\sf Flat}(c_1 c_2 \dots c_n).
\]
\end{lemma}

\begin{proof}
Let \( {\sf Flat}(a_1 a_2 \dots a_n) = b_1 b_2 \dots b_n \). By definition, if \( a_i < a_j \), then
\[
\left\{ k \mid a_k \leq a_i \right\} \sqcup \{a_j\} \subseteq \left\{ k \mid a_k \leq a_j \right\},
\]
so \( b_i < b_j \). Similarly, if \( a_i > a_j \), then \( b_i > b_j \). Hence,
\[ a_1 a_2 \dots a_n \simeq b_1 b_2 \dots b_n={\sf Flat}(a_1 a_2 \dots a_n).\]
The second claim follows immediately. \qedhere
\end{proof}

Now, for positive integer \( n \), we define a map \( \varphi_n \colon S_{n+1} \to S_n \) by
\[
\varphi_n(w_1 w_2 \dots w_n w_{n+1}) \coloneqq {\sf Flat}(w_1 w_2 \dots w_n).
\]

\begin{example}
Let \( w = 14352 \in S_5 \). Then,
\[
\varphi_5(14352) = {\sf Flat}(1435) = 1324 \in S_4.
\]
\end{example}

In the following discussion, for each $k \in \mathbb{Z}$, denote $k^* \coloneqq k+\frac{1}{2}$.

\begin{proposition}
\label{prop:relative-order}
If $\varphi_n(w_1 w_2 \dots w_n w_{n+1})=v_1v_2 \dots v_n$, then
\[w_1 w_2 \dots w_n w_{n+1} \simeq v_1v_2 \dots v_n \, k^*,\] where  $k=w_{n+1}-1$.
\end{proposition}

\begin{proof}
Lemma~\ref{lemma:flat} implies that $w_1 w_2 \dots w_n \simeq v_1v_2 \dots v_n$. As a result,
\[ v_i =
\begin{cases}
w_i, & \text{if } w_i < w_{n+1}, \\
w_i - 1, & \text{if } w_i > w_{n+1},
\end{cases}
\]
for any $i \in [n]$. Therefore, if $w_i<w_{n+1}$, then $v_i < k^*$. Conversely, if $w_i>w_{n+1}$, we have $v_i > k^*$. Combining these with the fact that
\[w_1w_2 \dots w_n \simeq v_1v_2 \dots v_n,\]
we obtain the conclusion. \qedhere
\end{proof}

\begin{lemma}
\label{lemma:inverse}
For any positive integer \( n \) and any \( v \in S_n \), we have
\[
\left| \varphi_n^{-1}(v) \right| = n+1.
\]
\end{lemma}

\begin{proof}
By Proposition~\ref{prop:relative-order}, each permutation $w \in \varphi_n^{-1}(v)$ is of the form ${\sf Flat}(v \, k^*)$ for $k =w_{n+1}-1 \in \{0,1, \dots, n\}$. Conversely, since both \( {\sf Flat} \) and \( \varphi_n \) preserve the relative order of the first \( n \) entries, we have
\[
\varphi_n\left( {\sf Flat}(v \, k^*) \right) = {\sf Flat}(v) = v \quad \text{for all} \quad k \in \{0, 1, \dots, n\}.
\]
Thus, \( {\sf Flat}(v \, k^*) \in \varphi_n^{-1}(v) \) for all \( k \in \{0, 1, \dots, n\} \). Since  the last entry of \( {\sf Flat}(v \, k^*)\) equals to $k+1$, those are \( n+1 \) distinct permutations in \( S_{n+1} \). Hence, \( |\varphi_n^{-1}(v)| = n+1 \). \qedhere
\end{proof}

\begin{example}
Consider \( v = 3124 \in S_4 \). To compute \( \varphi_4^{-1}(v) \), we append each of the values \( 0^*, 1^*, 2^*, 3^*, 4^* \) to the end of \( v \), and then apply {\sf Flat} to each resulting \( 5 \)-word. This yields the following five elements in \( \varphi_4^{-1}(v) \subseteq S_5 \):
\[
\begin{aligned}
{\sf Flat}(3124\,0^*) &= 42351, \\
{\sf Flat}(3124\,1^*) &= 41352, \\
{\sf Flat}(3124\,2^*) &= 41253, \\
{\sf Flat}(3124\,3^*) &= 31254, \\
{\sf Flat}(3124\,4^*) &= 31245.
\end{aligned}
\]
\end{example}

Recall from Section~\ref{subsection Schensted} that \( V_n \) is the set of all permutations in \( S_n \) whose Schensted insertion contains no lateral bumps.

\begin{lemma}
\label{lemma restriction}
For any positive integer \( n \), if \( w \in V_{n+1} \), then \( \varphi_n(w) \in V_n \).
\end{lemma}

\begin{proof}
Let \( w = w_1 w_2 \dots w_{n+1} \in V_{n+1} \), and suppose \( \varphi_n(w) = v_1 v_2 \dots v_n \in S_n \). Let \( k = w_{n+1} - 1 \). By Proposition~\ref{prop:relative-order}, we have
\[
v_1 v_2 \dots v_n\, k^* \simeq w_1 w_2 \dots w_{n+1}.
\]
Since lateral bumps in Schensted insertion depend only on the relative order of entries, the absence of lateral bumps in \( w_1w_2 \dots w_{n+1} \) implies their absence in \( v_1v_2 \dots v_n \), i.e., 
\[ \varphi_n(w)= v_1 v_2 \dots v_n \in V_n . \qedhere\]
\end{proof}

Since \( \varphi_n(V_{n+1}) \subseteq V_n \) by Lemma~\ref{lemma restriction}, we can define the map \( \psi_n \) as the restriction of \( \varphi_n \) to the domain \( V_{n+1} \), that is, \( \psi_n \colon V_{n+1} \to V_n \). Considering all the maps \( \psi_n \) for \( n \in \mathbb{N} \), we obtain a \emph{rooted tree}, where the root is the identity permutation in \( S_1 \). (We adopt graph-theoretic terminology for convenience.) The rooted tree is illustrated in Figure~\ref{fig:tree}. For each \( w \in V_n \), we view the elements in \( \psi_n^{-1}(w) \subseteq V_{n+1} \) as the \emph{children} of \( w \). By Lemma~\ref{lemma:inverse}, we have
\begin{equation} \label{eq 1}
    |\psi_n^{-1}(w)| \leq n+1,
\end{equation}
and consequently,
\begin{equation} \label{eq 2}
    |V_{n+1}| \leq (n+1)\,|V_n|.
\end{equation}

\begin{figure}
    \centering
    \ytableausetup{boxsize=1.5em}
    \begin{tikzpicture}[
         every node/.style={anchor=center, font=\small, align=center},
         level distance=2.5cm,  
         level 1/.style={sibling distance=5cm},  
         level 2/.style={sibling distance=2cm}, 
         edge from parent/.style={draw, <-}
     ]
        \node at (-6, 0) {$V_1$};  
        \node at (-6, -2.5) {$V_2$};    
        \node at (-6, -5) {$V_3$};   

        \node at (-6, -6.5) {$\vdots$};
        \node at (-3.5, -6.5) {$\vdots$};
        \node at (-1.5, -6.5) {$\vdots$};
        \node at (0.5, -6.5) {$\vdots$};
        \node at (2.5, -6.5) {$\vdots$};
        \node at (4.5, -6.5) {$\vdots$};
        
         \node { 
             $w=1$\\[5pt]
             $\begin{ytableau}
                 1
             \end{ytableau}$
         }
         child { 
             node {
                 $w=12$\\[5pt]
                 $\begin{ytableau}
                     1 & 2
                 \end{ytableau}$
             }
             child { 
                 node {
                     $w=123$\\[5pt]
                     $\begin{ytableau}
                         1 & 2 & 3
                     \end{ytableau}$
                 }
             }
             child { 
                 node {
                     $w=231$\\[5pt]
                     $\begin{ytableau}
                        2\\
                         1 & 3
                     \end{ytableau}$
                 }
             }
         }
         child { 
             node {
                 $w=21$\\[5pt]
                 $\begin{ytableau}
                     2\\
                     1
                 \end{ytableau}$
             }
             child { 
                 node {
                     $w=321$\\[5pt]
                     $\begin{ytableau}
                         3\\
                         2\\
                         1
                     \end{ytableau}$
                 }
             }
             child { 
                 node {
                     $w=312$\\[5pt]
                     $\begin{ytableau}
                        3\\
                         1 & 2
                     \end{ytableau}$
                 }
             }
             child { 
                 node {
                     $w=213$\\[5pt]
                     $\begin{ytableau}
                        2\\
                         1 & 3
                     \end{ytableau}$
                 }
             }
         };
        
     \end{tikzpicture}
    \caption{The rooted tree of permutations with no lateral bumps.}
    \label{fig:tree}
\end{figure}
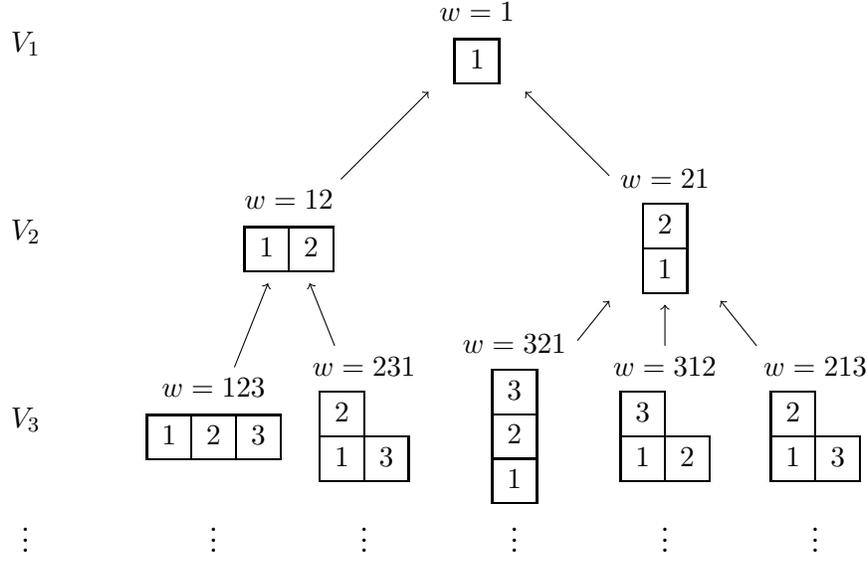

For \( \lambda \in {\sf Par}(n) \), suppose that \( \lambda = (\lambda_1, \dots, \lambda_k) \). We define \( L(\lambda) \coloneqq \lambda_1 \) as the length of the first row of the shape \( \lambda \). For example, for \( \lambda = (2, 2, 1) \), we have \( L(\lambda) = 2 \). Recall the Plancherel measure from Section~\ref{subsection limit shape theorem}, and we state the following lemma.

\begin{lemma}
\label{L(lambda) lemma}
Under the Plancherel measure \( \mathbb{P}_n^* \) on \( {\sf Par}(n) \), we have
\[
\mathbb{P}_n^*\left(L(\lambda) \geq \sqrt{2n}\right) \longrightarrow 1 \quad \text{as} \quad n \longrightarrow \infty.
\]
\end{lemma}

\begin{proof}
Substituting \( \epsilon = 1 - \frac{\sqrt{2}}{2} \) in Theorem~\ref{limit shape theorem}, which implies
\[
\mathbb{P}_n^*\left( \frac{\sqrt{2}}{2} \Gamma \subseteq \frac{1}{\sqrt{n}} S_{\lambda} \right) \longrightarrow 1 \quad \text{as} \quad n \longrightarrow \infty.
\]
If \( \frac{\sqrt{2}}{2} \Gamma \subseteq \frac{1}{\sqrt{n}} S_{\lambda} \), then \( L(\lambda) \geq \sqrt{2n} \).
Therefore, we have
\[
\mathbb{P}_n^*\left( L(\lambda) \geq \sqrt{2n} \right) \geq \mathbb{P}_n^*\left( \frac{\sqrt{2}}{2} \Gamma \subseteq \frac{1}{\sqrt{n}} S_{\lambda} \right).
\]
By the squeeze theorem, it follows that
\[
\mathbb{P}_n^*\left( L(\lambda) \geq \sqrt{2n} \right) \longrightarrow 1 \quad \text{as} \quad n \longrightarrow \infty. \qedhere
\]
\end{proof}

\begin{lemma} 
\label{same height}
For any \( \lambda \in {\sf Par}(n) \), if \( L(\lambda) \geq \sqrt{2n} \), then the shape \( \lambda \) has two columns of the same height.
\end{lemma}

\begin{proof}
Suppose all \( L(\lambda) \) columns of \( \lambda \) have distinct heights. Then the total number of boxes in \( \lambda \) is at least the sum of the \( L(\lambda) \) smallest distinct positive integers:
\[
|\lambda| \geq 1 + 2 + \cdots + L(\lambda) = \frac{1}{2} L(\lambda)\left(L(\lambda)+1\right).
\]
But since \( |\lambda| = n \), this implies
\[
n \geq \frac{1}{2} L(\lambda)\left(L(\lambda)+1\right) > \frac{1}{2} L(\lambda)^2.
\]
Hence, $ L(\lambda) < \sqrt{2n}$. This contradicts the assumption that \( L(\lambda) \geq \sqrt{2n} \), so \( \lambda \) must have at least two columns of the same height. \qedhere
\end{proof}

\begin{figure}
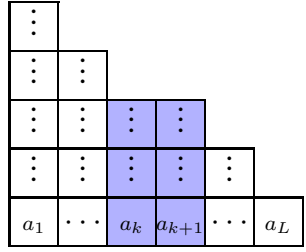

\ytableausetup{boxsize=1.5em}
    \centering
    \[
    \begin{ytableau}
        \vdots \\
        \vdots & \vdots \\
        \vdots & \vdots & *(blue!30) \vdots & *(blue!30) \vdots\\
        \vdots & \vdots & *(blue!30) \vdots & *(blue!30) \vdots & \vdots\\
        {\scriptstyle a_1}  & \cdots  & *(blue!30) {\scriptstyle a_k} & *(blue!30) {\scriptstyle a_{k+1}} & \cdots & {\scriptstyle a_{L}}
    \end{ytableau}
    \]
    \caption{Example $\lambda$ with 2 columns, $k$ and $k+1$, of the same height.}
    \label{fig:2column}
\end{figure}

\begin{lemma}
\label{at most n children lemma}
For any \( v \in V_n \), if the shape \( \lambda_v \coloneqq {\sf sh}(v) \) has two columns with the same height, then $\left| \psi_n^{-1}(v) \right| \leq n$, i.e., $v$ has at most $n$ children.
\end{lemma}

\begin{proof}
Suppose that \( k^{\text{th}} \) and \( (k+1)^{\text{th}} \) columns of \( \lambda_v \coloneqq {\sf sh}(v) \) have the same height (see Figure~\ref{fig:2column}). Let \( a_1, a_2, \dots, a_L \) be the entries in the first row of \( \lambda_v \), and define
\[ w \coloneqq {\sf Flat}(v \, a_k^*) \in S_{n+1}.\]
In the final insertion of \( v \, a_k^* \), i.e., inserting \( a_k^* \) into the tableau \( \lambda_v \), the number \( a_k^* \) replaces \( a_{k+1} \), bumping \( a_{k+1} \) to the second row. If there are consecutive vertical bumps in the \( (k+1)^{\text{th}} \) column, then the bumping process will get stuck at the top of the column. Therefore, the final insertion causes a lateral bump. Since lateral bumps depend only on the relative order of entries, the Schensted insertion of \( w \) also undergoes a lateral bump. Thus, \( \left| \psi_n^{-1}(v) \right| \leq n \), meaning \( v \) has at most \( n \) children. \qedhere
\end{proof}

\begin{lemma}
\label{lim lemma}
\[
\lim_{n \to \infty}\left(\frac{1}{2} \cdot \frac{3}{4} \cdots \frac{2n-1}{2n} \right) = 0.
\]
\end{lemma}

\begin{proof}
Let
\[
a_n = \frac{1}{2} \cdot \frac{3}{4} \cdots \frac{2n-1}{2n} = \frac{(2n)!}{2^{2n} (n!)^2}.
\]
We recall Stirling's approximation
\[
\lim_{n \to \infty} \frac{\sqrt{2 \pi n} \left(\frac{n}{e}\right)^n}{n!}=1.
\]
By applying the approximation, we obtain
\[
\lim_{n \to \infty} a_n = \lim_{n \to \infty} \frac{\sqrt{4 \pi n} \left(\frac{2n}{e}\right)^{2n}}{2^{2n} (2 \pi n) \left(\frac{n}{e}\right)^{2n}} = \lim_{n \to \infty} \frac{1}{\sqrt{\pi n}} = 0. \qedhere
\]
\end{proof}

\begin{proof}[\textbf{Proof of Theorem~\ref{main thm}}]

Define \( p_n \coloneqq \frac{|V_n|}{n!} \). For all \( n \in \mathbb{N} \), using~\eqref{eq 2}, we have
\[
p_{n+1} = \frac{|V_{n+1}|}{(n+1)!} \leq \frac{(n+1)|V_n|}{(n+1)!} = p_n.
\]
Thus, the sequence \( \{p_n\}_{n=1}^{\infty} \) is monotone decreasing and bounded below by \( 0 \). Therefore, it converges to some limit \( \epsilon \geq 0 \). Suppose \( \epsilon > 0 \). By Lemma~\ref{L(lambda) lemma}, for any \( \delta > 0 \), there exists \( N \in \mathbb{N} \) such that
\[
\mathbb{P}_n^*\left(L\left(\lambda \right) \geq \sqrt{2n}\right) > 1 - \delta, \quad \text{for all} \quad n \geq N.
\]
By the relation of the measures \( \mathbb{P}_n \) on \( S_n \) and \( \mathbb{P}_n^* \) on \( {\sf Par}(n) \) in Section~\ref{subsection limit shape theorem}, we obtain
\[
\mathbb{P}_n\left(L\left(\lambda_w \right) \geq \sqrt{2n}\right) =
\mathbb{P}_n^*\left(L\left(\lambda \right) \geq \sqrt{2n}\right) > 1 - \delta, \quad \text{for all} \quad n \geq N.
\]
We denote \( H_n \) as the set of permutations \( w \in S_n \) such that \( \lambda_w \) has columns of different heights. By Lemma~\ref{same height}, if \( L(\lambda_w) \geq \sqrt{2n} \), then \( \lambda_w \) has two columns of the same height, i.e., \( w \notin H_n \). Therefore, we obtain
\[
\mathbb{P}_n\left( w \notin H_n \right) \geq
\mathbb{P}_n\left(L\left(\lambda_w \right) \geq \sqrt{2n}\right) >
1 - \delta, \quad \text{for all} \quad n \geq N.
\]
This implies that
\[
\mathbb{P}_n\left( w \in H_n \cap V_n \right) \leq \mathbb{P}_n\left( w \in H_n \right) < \delta, \quad \text{for all} \quad n \geq N.
\]
We denote \( U_n \) as the set of permutations in \( V_n \) that have \( n+1 \) children. By Lemma~\ref{at most n children lemma}, if \( w \in U_n \), then \( w \in H_n \). Thus, we obtain
\[
|U_n| \leq |V_n \cap H_n| = n! \times \mathbb{P}_n\left( w \in H_n \cap V_n \right) < n! \times \delta.
\]
Choosing \( \delta = \frac{1}{2} \epsilon \), we have
\[
|U_n| < \delta n! = \frac{1}{2} \epsilon n! \leq \frac{1}{2} |V_n|,
\]
because \( \frac{|V_n|}{n!} = p_n \geq \epsilon \). By~\eqref{eq 1} and the definition of $U_n$, any \( w \in V_n \setminus U_n \) has at most \( n \) children. Thus,
\[
|V_{n+1}| \leq n \cdot |V_n \setminus U_n| + (n+1) \cdot |U_n| \leq \left(n + \frac{1}{2} \right) |V_n|.
\]
Therefore, we obtain
\[
p_{n+1} \leq \frac{n + \frac{1}{2}}{n+1} \, p_n = \left(1 - \frac{1}{2(n+1)}\right) p_n.
\]
Since this inequality holds for all \( n \geq N \), we can apply it repeatedly as follows:
\[
\begin{aligned}
p_n & \leq \left(1 - \frac{1}{2n} \right) p_{n-1} \\
    & \leq \left(1 - \frac{1}{2n} \right) \left(1 - \frac{1}{2(n-1)} \right) p_{n-2} \\
    & \quad \vdots \\
    & \leq \left(1 - \frac{1}{2n} \right) \left(1 - \frac{1}{2(n-1)} \right) \cdots \left(1 - \frac{1}{2(N+1)} \right) p_N \\
    & = \prod_{i=N+1}^{n} \left(1 - \frac{1}{2i} \right) p_N.
\end{aligned}
\]
From Lemma~\ref{lim lemma}, we obtain
\[
p_n \leq \prod_{i=N+1}^{n} \left(1 - \frac{1}{2i} \right) p_N = \frac{a_n}{a_N} p_N \longrightarrow 0 \quad \text{as} \quad n \longrightarrow \infty.
\]
Therefore, by the squeeze theorem, we conclude that
\[
\lim_{n \to \infty} p_n = 0.
\]
This leads to a contradiction, as we initially assumed \( \lim_{n \to \infty} p_n = \epsilon > 0 \). Hence, we must have \( \epsilon = 0 \). \qedhere
\end{proof}

\section*{Acknowledgment}
We sincerely thank Alexander Yong for his valuable discussions and significant contributions to the editing of this paper. We are also grateful to Victor Reiner for his insightful discussions, particularly during his visit to UIUC. David Xia was supported by the ICLUE program, funded through the NSF RTG grant DMS-1937241.

\bibliographystyle{abbrv}
\bibliography{bibliography.bib}

\end{document}